\documentclass[a4paper,twoside,fleqn,10pt]{article}

\pdfoutput=1

\usepackage{epsfig}
\usepackage{enumerate}
\usepackage{amsmath}
\usepackage{amssymb}
\usepackage{amsthm}

\newtheorem{Theorem}{Theorem}[section]
\newtheorem{Lemma}[Theorem]{Lemma}
\newtheorem{Proposition}[Theorem]{Proposition}
\newtheorem{Corollary}[Theorem]{Corollary}
\newtheorem{Definition}[Theorem]{Definition}
\newtheorem{Remark}[Theorem]{Remark}

\newenvironment{theorem}{\begin{Theorem} \begin{sl}}{\end{sl}
                         \end{Theorem}}
\newenvironment{lemma}{\begin{Lemma} \begin{sl}}{\end{sl} \end{Lemma}}
\newenvironment{proposition}{\begin{Proposition}
        \begin{sl}}{\end{sl} \end{Proposition}}

\newenvironment{definition}{\begin{Definition}
       \begin{sl}}{\end{sl} \end{Definition}}
\newenvironment{remark}{\begin{Remark}
       \begin{rm}}{\end{rm} \end{Remark}}

\newcommand{\C}{{\mathbb C}}
\newcommand{\T}{{\mathbb T}}
\newcommand{\cX}{\mathcal{X}}
\newcommand{\degree}{\text{degree}}
\newcommand{\partiel}[2]{\frac{\partial #1}{\partial #2}}

\newcommand{\R}{{\mathbb R}}
\newcommand{\ad}{\text{ad}}
\newcommand{\bvv}[1]{\frac{\partial}{\partial #1}}
\newcommand{\cO}{\mathcal{O}}
\newcommand{\cinf}{C^{\infty}}
\newcommand{\commentaar}[1]{{}}
\newcommand{\eps}{\varepsilon}
\newcommand{\expad}{\text{expad}}
\newcommand{\e}{\text{e}}
\newcommand{\hot}{\text{h.o.t.}}

\newcommand{\system}[1]{\left\{\begin{aligned} #1 \end{aligned}\right.}
\newcommand{\wpq}{\tfrac{p}{q}}

\newcommand{\version}{\today}
\newcommand{\shorttitle}{Stability pockets...seasonality}
\newcommand{\shortauthor}{Hoveijn}

\title{Stability pockets of a periodically forced oscillator in a model for seasonality}
\author
{
{\protect\normalsize Igor Hoveijn} \\
{\protect\footnotesize\protect\it 
University of Groningen, Faculty of Mathematics and Natural Sciences, The Netherlands}
}
\date{\version}

\commentaar{
- meer referenties
}

\begin{document}

\maketitle

%\tableofcontents

%--------------------------------------------------------------------------------

\begin{abstract}
\noindent
A periodically forced oscillator in a model for seasonality shows stability pockets and chains thereof in the parameter plane. The frequency of the oscillator and the season indicated by a value between zero and one are the two parameters. The present study is intended as a theoretical complement to the numerical study of Schmal et al. \cite{smhb} of stability pockets or Arnol'd onions in their terminology. We construct the Poincar\'e map of the forced oscillator and show that the Arnol'd tongues are taken into stability pockets by a map with a number of folds. Stability pockets are already observed in an article by van der Pol \& Strutt in 1928, see \cite{ps1928} and later explained by Broer \& Levi in 1995, see \cite{bl1995}.
\end{abstract}

\textbf{Keywords:} forced oscillator, resonance, synchronization, Poincar\'e map, stability pocket, circadian clock, Zeitgeber, entrainment

%--------------------------------------------------------------------------------
\section{Introduction}\label{sec:intro}
The numerical study by Schmal et al. \cite{smhb} of a model for seasonal effects on the circadian clock shows stability pockets and chains thereof in the parameter plane. This study aims to complement their results by giving a theoretical background for the observed phenomena. We also indicate what is to be expected when their system is perturbed. The setting of the problem is bifurcations of parameter dependent dynamical systems, in particular periodically forced oscillators.

A periodically forced oscillator shows periodic dynamics or synchronization if the frequency of the forcing is close enough to a rational multiple of the frequency of the oscillator in absence of forcing. On the other hand, if this frequency ratio is not close enough to a rational number for many frquency ratio's the dynamics is quasi periodic. Now we consider the periodically forced oscillator as a system of two asymmetrically coupled oscillators, where the first is forced by the second. The latter has fixed dynamics. In particular its frequency is fixed and we set it equal to one by a scaling of time. Thus the frequency of the first oscillator becomes an essential parameter of the system. Another important parameter is the coupling strength. In the parameter plane of coupling strength versus frequency there are regions, called Arnol'd tongues (see \cite{arn1983}) or resonance regions, where the first oscillator synchronizes with the second, that is there are one or more stable (and unstable) periodic solutions. These tongues are wedge shaped and their vertices lie at rational points on the frequency axis. So for a fixed value of the coupling strength there is a frequency interval where synchronization occurs. On the boundary of this interval the periodic solutions disappear in saddle-node bifurcations. If the coupling strength goes to zero the interval shrinks to a (rational) point. If the coupling strength increases there may be many other bifurcations, for example period doubling bifurcations, see \cite{bst1998}.

In certain examples of periodically forced oscillators the tongues close again in a second vertex, forming a so called called \emph{stability pocket} (called \emph{Arnol'd onion} by Schmal et al. \cite{smhb}), see left part of figure \ref{fig:mtpocket}. This phenomenon occurs for example in Hill's equation and it has already been observed by van der Pol \& Strutt \cite{ps1928}, although they do not mention it explicitly as such. Much later Hill's equation has been reanalyzed by Broer \& Levi \cite{bl1995} using methods not available to van der Pol \& Strutt. Then the term \emph{instability pocket} was introduced. However, it depends on the point of view whether a pocket is called a stability or an instability pocket, see remark \ref{rem:pockets}. Although the context of the Hill equation differs from ours, the mechanism by which the pockets are formed is the same. In both systems a general underlying system has a wedge shaped resonance region. The map that takes the parameter space of the specific example to the parameter space of the general system has one or more folds, thus generating pockets.

The model in Schmal et al. \cite{smhb} consists of a differential equation for a specific two dimensional oscillator, in fact the normal form of the Hopf bifurcation, with periodic forcing. Here we take a slightly different approach. Our oscillator will be a general phase oscillator asymmetrically coupled to a second phase oscillator with constant frequency, the forcing oscillator or Zeitgeber. First we study a general coupling and approximate the Poincar\'e map. Then we use these results to study seasonality like in \cite{smhb} by considering a periodic block function as forcing, where the length of the block (modeling the length of the daylight interval and thus season) is controlled by an additional parameter. 

We could have started with a general oscillator in two dimensions with an external periodic forcing. But then the technicalities of bringing the system in manageable form and keeping it so when the external forcing is applied would obscure the phenomenon we want to study. Our primary goal is the way the dynamics in the phase direction of the oscillator changes when the external force is applied: we are less interested in the change of the shape and position of the periodic orbit. Therefore we start with a phase oscillator from the very beginning.

\commentaar{opmerking over de term 'forcing', waar die vandaan komt, hier iets heel anders is en toch zo heet}
%--------------------------------------------------------------------------------

%--------------------------------------------------------------------------------
\section{Two asymmetrically coupled phase oscillators}\label{sec:tpo}
Let $\psi$ and $\varphi$ be phase angles so $\psi, \varphi \in S^1$ where $\psi$ is interpreted as the phase angle of the oscillator and $\varphi$ as the phase angle of the external forcing. Furthermore, let $f_{\mu}$ be a smooth function on $\T^2$. The following asymmetrically coupled system describes a periodically forced phase oscillator.
\begin{equation}\label{eq:tpo}
\system{\dot{\psi} &= \omega + \mu f_{\mu}(\psi, \varphi) \\ \dot{\varphi} &= 1}
\end{equation}
The system depends on small parameters $\mu, \delta,\cdots \in \R^m$. The frequency of the oscillator is $\omega$ for $\mu = 0$. In that case the dynamics of the system is quasi-periodic if $\omega$ is irrational, but if $\omega$ is rational, the dynamics is periodic. Thus for $\mu = 0$ periodic dynamics is exceptional and quasi periodic dynamics is typical. This changes dramatically if $\mu \neq 0$, then periodic dynamics becomes typical. Thus it seems natural to take a closer look near rational values of $\omega$, therefore we set $\omega = \wpq + \delta$, with $p$ and $q$ relative prime integers and $\delta$ is a small parameter. It is most efficient to study this system via a Poincar\'e map. A good candidate is the map that scores $\psi$ at consecutive crossings of $\varphi = 0$, sometimes called the stroboscopic map. Since all Poincar\'e maps are equivalent, results will not depend on this choice. For convenience we switch to the lift of this system to $\R^2$, then the differential equation becomes
\begin{equation}\label{eq:tpolift}
\system{\dot{x} &= \wpq + \delta + \mu f_{\mu}(x, y) \\ \dot{y} &= 1}
\end{equation}
Now we define the (lift of the) Poincar\'e map as follows.
\begin{definition}\label{def:tpopmap}
Let $\Phi_t$ be the flow of equation \eqref{eq:tpolift}, then the Poincar\'e map $F$ is implicitly defined by $(F(x), 1) = \Phi_1(x, 0)$.
\end{definition}
The map $F$ in this definition is the lift of a circle map $f$ (not to be confused with $f_{\mu}$), the Poincar\'e map on the circle. In actual computations it is easier to use the lift $F$. Now since $F$ is the lift of a circle map it must be of the form $F(x) = x + \omega + \mu h_{\mu}(x)$ where $h$ is a 1-periodic function. Unfortunately there is no easy connection between $h$ and $f$ in equation \eqref{eq:tpolift}. But by a so called normal form transformation we obtain a vectorfield approximation of the Poincar\'e map $F$, see section \ref{sec:proofs}. The approximating vectorfield is in the righthand side of the following equation
\begin{equation}\label{eq:vvapprox}
\dot{u} = \delta + \sum \tilde{f}_{k,l,m}\,\mu^m\,\e^{2\pi iku}
\end{equation}
where $\tilde{f}_{k,l,m}$ depends in a complicated way on the Taylor-Fourier coefficients of $f_{\mu}$ and the sum runs over all $k$ and $l$ with $pk+ql=0$ and $m \in \{1,\cdots,n\}$. To be more precise, the time-1 flow of this vector field approximates the Poincar\'e map. This is formulated more precisely in section \ref{sec:vvapprox}.

The main property of the Poincar\'e map $f$ we use is that a stationary or a periodic point of $f$ corresponds to a periodic solution of equation \eqref{eq:tpo} and vice verse. Thus the existence of periodic solutions of equation \eqref{eq:tpo} can be read off from the existence of stationary or periodic points of $f$. The vectorfield approximation does not cover the full dynamics of the original system. But hyperbolic stationary points of \eqref{eq:vvapprox} correspond to stationary points of the Poincar\'e map $F$. Moreover by persistence of saddle-node bifurcations we recover the well-known picture of resonance tongues for $F$ in the $(\omega, \mu)$-plane from equation \eqref{eq:vvapprox}. For each pair $(p,q)$ a tongue emanates from the $\omega$-axis at $\omega = \frac{p}{q}$. The tongue at $(p,q) = (1,1)$ is called the \emph{main tongue}, the others are just labeled by the pair $(p,q)$. Indeed we may rewrite the differential equation as follows
\begin{equation}\label{eq:vvapprox2}
\dot{u} = \delta + \mu \tilde{f}_{\mu}(u)
\end{equation}
where $\tilde{f}_{\mu}$ is a 1-periodic function. Then at least near $(0,0)$ stationary points exist in a region in the $(\delta, \mu)$-plane bounded by curves of saddle-node bifurcations of the form $\mu = \gamma_1 \delta$ and $\mu = \gamma_2 \delta$, for some constants $\gamma_1$ and $\gamma_2$. In the simplest case where $\tilde{f}_{\mu}(u) = \sin(2\pi u)$ we have $\mu = \pm \delta$.

In the next section we take a specific form for equation \eqref{eq:tpolift} containing parameters $\sigma$ and $\lambda$ and we study the map $(\sigma, \lambda) \mapsto (\delta, \mu)$. The inverse image of this map takes the tongues of the general equation for two asymmetrically coupled oscillators to the resonance regions of model with season dependent forcing.
%--------------------------------------------------------------------------------

%--------------------------------------------------------------------------------
\section{Main results}\label{sec:mainres}

%--------------------------------------------------------------------------------
\subsection{Stability pockets in the seasonal oscillator}\label{sec:season}
In a simple model of season dependent forcing we take a phase oscillator with a particular 1-periodic forcing. We will call it the \emph{seasonal oscillator} for short. The general form will be the following.
\begin{equation}\label{eq:season}
\system{\dot{\psi} &= \omega + \eta f(\psi) + \eps g_{\lambda}(\varphi) \\ \dot{\varphi} &= 1}
\end{equation}
Here $f$ and $g_{\lambda}$ are $\cinf$ functions on $S^1$ so that the corresponding vector field is also $\cinf$. The function $f$ describes the non-linearity of the oscillator and $g_{\lambda}$ determines the external periodic forcing. As before we take $\omega = \tfrac{p}{q} + \sigma$, with positive integers $p$ and $q$. The parameters $\sigma$, $\eta$ and $\eps$ are small, but not necessarily of the same order.

To model seasonality we use a function as in \cite{smhb}, namely one that depends on a non-small parameter $\lambda$, which determines the fraction of the period the forcing is 'on'. Thus we interpret $\lambda = 0$ as 'winter' and $\lambda = 1$ as 'summer'. A simple example being a block function with a block of length $\lambda \in [0,1]$
\begin{equation}\label{eq:block}
g_{\lambda}(t) = 
\begin{cases}
1, & 0 \leq t < \lambda\\
0, & \lambda \leq t < 1
\end{cases}
\end{equation}
with $g_{\lambda}(t+1) = g_{\lambda}(t)$ for all $t$. This function is not $\cinf$. Therefore we replace $g_{\lambda}$ by the convolution $\phi * g_{\lambda}$ where $\phi$ is for example the function $\phi(x) = \sqrt{\frac{\alpha}{\pi}} \exp(-\alpha x^2)$. In section \ref{sec:fcf} we will show that the result does not depend on the choice of $\phi$. For simplicity we assume from now on that $g_{\lambda}$ is $\cinf$.

We use the methods of the previous section to find a vectorfield approximation of the Poincar\'e map for the seasonal oscillator. The following theorem gives a more precise statement. For a proof see section \ref{sec:apmapso} where also the meaning of \emph{degree} will be clarified.
\begin{theorem}\label{the:season}
The vectorfield approximation up to degree two of the Poincar\'e map of equation \eqref{eq:season} is 
\begin{equation}\label{eq:apmapseasonal}
%\dot{u} = \sigma - \frac{q}{p} \Big(\eta^2 \sum_k f_k f_{-k} + \eps \eta \sum_{pk+ql=0} f_k g_l \e^{2\pi iku} \Big)
\dot{u} = \sigma - \frac{q}{p} \Big(\eta^2 \sum_k f_k f_{-k} + \eps \eta \sum_m f_{mq} g_{-mp}(\lambda) \e^{2\pi imqu} \Big)
\end{equation}
where $u = x - \frac{p}{q}y$. The Fourier coefficients of $f$ and $g_{\lambda}$ are $f_k$ and $g_k(\lambda)$ respectively, where
\begin{equation*}
g_k(\lambda) = \exp(-\frac{\pi^2 k^2}{\alpha}) \frac{i}{2\pi k} \big(\exp(-2\pi ik\lambda) - 1\big).
\end{equation*}
\end{theorem}
We wish to find a region in the parameter space of the seasonal oscillator, that is in the $(\sigma, \lambda)$-plane, where stationary points of equation \eqref{eq:apmapseasonal} exist. Since this equation is a special form of equation \eqref{eq:vvapprox} we consider the map $(\sigma, \lambda) \mapsto (\delta, \mu)$. The result is formulated in the following theorem.
\begin{theorem}\label{the:pockets}
Let $p$ and $q$ be relative prime as mentioned before. Then the map $(\sigma, \lambda) \mapsto (\delta, \mu)$ is implicitly given by:
\begin{equation*}
(\delta, \mu) = (\sigma - \eta^2 c_1 - \eps \eta \lambda c_2, \eps \eta |h(\lambda)|).
\end{equation*}
where $c_1$ and $c_2$ are constants and $h$ is a $\frac{p}{2}$-periodic function with zero average. This implies that the $(p,q)$-tongue has $p$ \emph{stability pockets}.
\end{theorem}
The map implicitly defined in the theorem takes the main tongue of a general periodically forced oscillator into a so called \emph{stability pocket} in the $(\sigma, \lambda)$-plane of the seasonal oscillator. Other tongues, depending on $p$, are mapped to a chain of stability pockets. In figure \ref{fig:pockets} a graphical representation of this result is shown.
\begin{figure}[htbp]
\setlength{\unitlength}{1mm}
\begin{picture}(60,65)(30,0)
\put(0,   0){\includegraphics[scale=0.8]{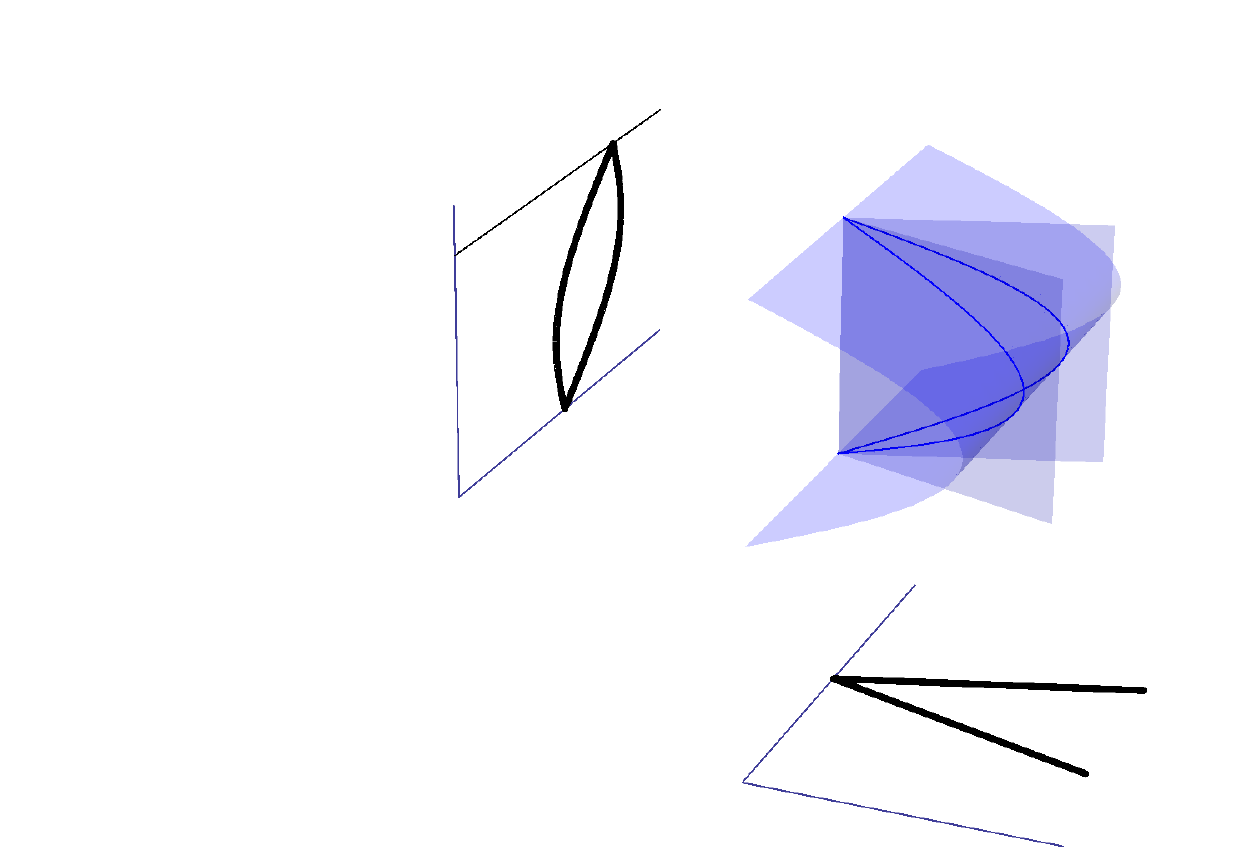}}
\put(63,  1){\includegraphics{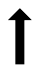}}
\put(35, 13){\includegraphics{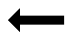}}
\put(83, -3){$\mu$}
\put(70, 20){$\delta$}
\put(53, 39){$\sigma$}
\put(33, 49){$\lambda$}
\end{picture}
\begin{picture}(0,0)(0,0)
\put(0,   0){\includegraphics[scale=0.8]{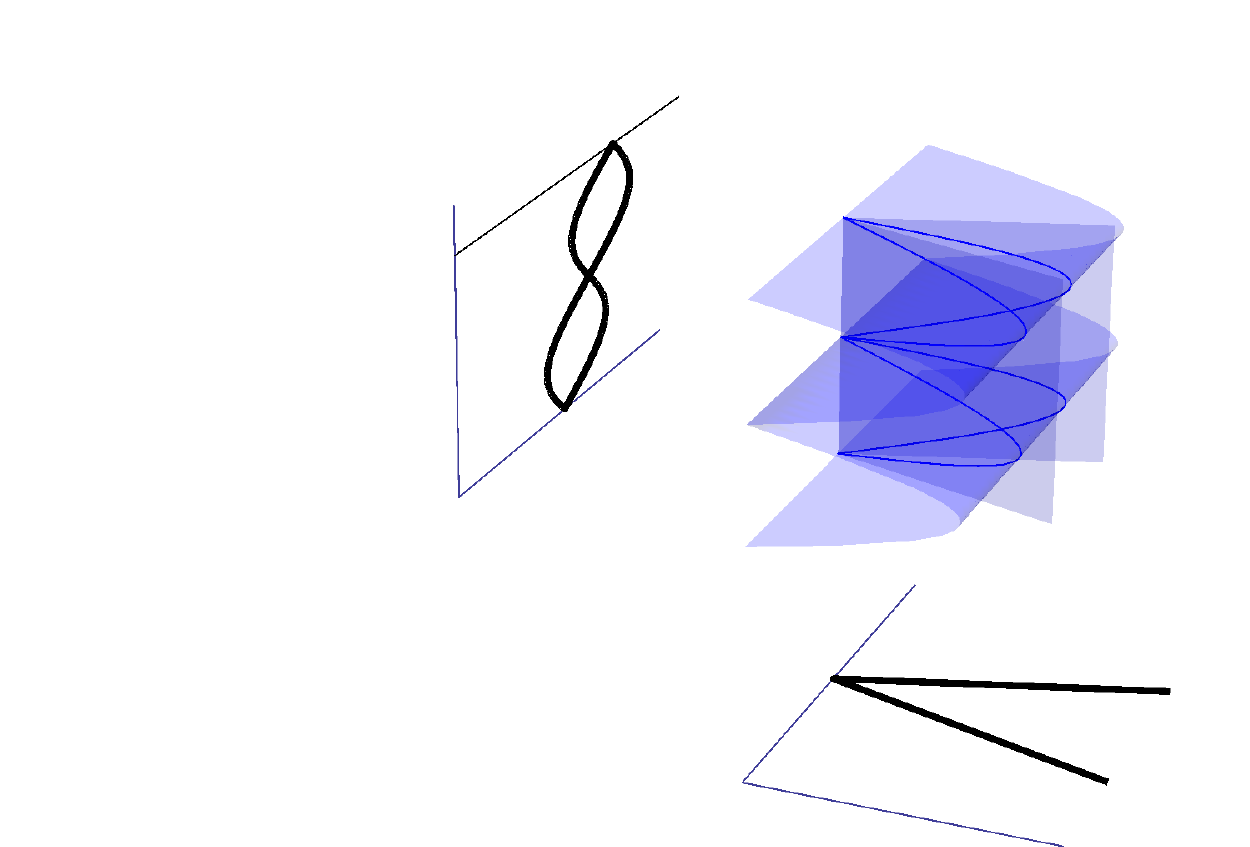}}
\put(63,  1){\includegraphics{vpijl.pdf}}
\put(35, 13){\includegraphics{hpijl.pdf}}
\put(83, -3){$\mu$}
\put(70, 20){$\delta$}
\put(53, 39){$\sigma$}
\put(33, 49){$\lambda$}
\end{picture}
\caption{\textit{Stability pockets and folds of the map $(\sigma, \lambda) \mapsto (\delta, \mu)$. Left: the pocket for the $(1,1)$-tongue; right: the chain of pockets for the $(2,1)$-tongue.}\label{fig:pockets}}
\end{figure}
\begin{remark}\label{rem:pockets}\strut
\begin{enumerate}[i)]\itemsep 0pt
\item Theorems \ref{the:season} and \ref{the:pockets} together support the numerical results in \cite{smhb}, there is an excellent qualitative agreement with their findings and the propositions. There is no reason to doubt that with some more effort this agreement can also be made quantitative.
\item In \cite{smhb} the stability pocket is called ``Arnol'd onion''. Here the term \emph{stability pocket} is used to connect with the existing literature on Hill's equation, see \cite{bl1995}. In Hill's equation the zero solution becomes unstable in the resonance tongues. In case such a tongue closes, a ``pocket'' is formed which is then called an instability pocket. In our case we are interested in synchronization, that is in the existence of stable periodic solutions. Since these exist in resonance tongues, a tongue closing and forming a ``pocket'' is now called a stability pocket.
\commentaar{Leg verband of geef verschil aan met Broer \& Levi, van der Pol \$ Strutt, Simo zal ook wel iets hebben.}
\item In the special case $f(\psi) = \sin(2\pi\psi)$ the map becomes simpler, namely $(\delta, \mu) = (\sigma - \eta^2 c_1 - \eps \eta \lambda c_2, \eps \eta \sqrt{2} |\sin(p\pi\lambda)|)$. In fact the map $(\delta, \mu) = (\sigma - \lambda, |\sin(p\pi\lambda)|)$ has essentially the same properties. This map is used to draw the graphs in figure \ref{fig:pockets}.
\item We have to find the map $(\sigma, \lambda) \mapsto (\delta, \mu)$ for each tongue. The reason is that no single parameterization of all normal forms exists: for each resonance $p:q$ we have to compute a new normal form.\commentaar{formulering}
\item In a more accurate model of seasonality, the parameter $\lambda$ will vary in time. Since this variation is slow with respect to the frequency of the forcing we may consider a model based on a slow-fast system. From this perspective we are studying the fast limit of such a system.
\end{enumerate}
\end{remark}
In a biological interpretation of these results, the image of the main tongue is perhaps the most relevant, see figure \ref{fig:mtpocket}. For a fixed value of $\lambda$ at a certain season, the \emph{range of entrainment} is indicated by a horizontal line intersecting the stability pocket. Only in this interval of $\omega$ values, the seasonal oscillator \emph{entraines} to the Zeitgeber. That is for these $\omega$ values, the oscillator has periodic solutions with the same period (or frequency) as the Zeitgeber. From the perspective of the oscillator the value of $\omega$ is fixed. Then the range of entrainment is in the seasonal direction. As can be seen in the figure, the oscillator cannot entrain to the Zeitgeber for all seasons. In case of the main tongue, not in the middle of summer and not in the middle of winter. When we look at the $(2,1)$-tongue, the range of entrainment in the seasonal direction may consist of two intervals, see figure \ref{fig:mtpocket}. For more biological interpretation and examples refer to \cite{smhb}. 
\begin{figure}[htbp]
\setlength{\unitlength}{1mm}
\begin{picture}(60,50)(20,5)
\put(0,   0){\includegraphics[scale=1.0]{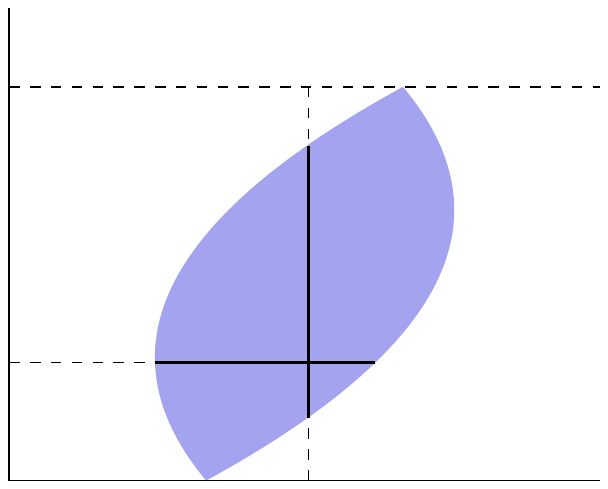}}
\put(26, 21){$\lambda$}
\put(59,  6){$\omega$}
\put(26,  8){$0$}
\put(26, 49){$1$}
\put(49,  6){$1$}
\end{picture}
\begin{picture}(0,0)(0,5)
\put(0,   0){\includegraphics[scale=1.0]{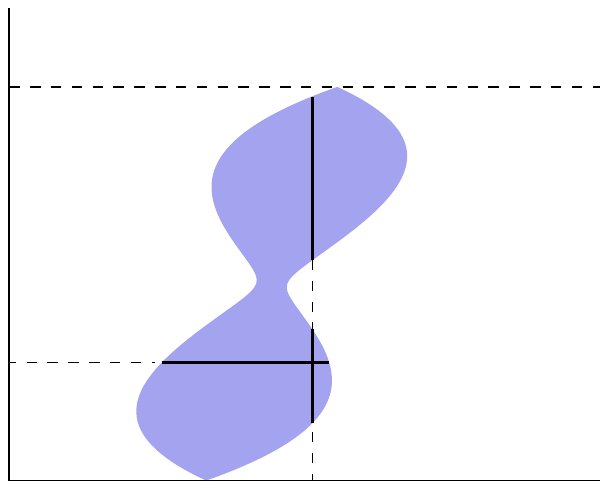}}
\put(26, 21){$\lambda$}
\put(59,  6){$\omega$}
\put(26,  8){$0$}
\put(26, 49){$1$}
\put(49,  6){$2$}
\end{picture}
\caption{\textit{Stability pockets and ranges of entrainment in both frequency and seasonal directions. The image of the main tongue is shown in the left figure. The right figure shows the image of the $(2,1)$-tongue for a small perturbation of $g_{\lambda}$.}\label{fig:mtpocket}}
\end{figure}
\begin{proof}[Proof of theorem \ref{the:pockets}]
Compare equation \eqref{eq:vvapprox2} to equation \eqref{eq:apmapseasonal} to determine the map $(\delta, \mu) \mapsto (\sigma, \lambda)$. We immediately see that $\delta = \sigma - \eta^2 c_1 - \eps \eta \lambda c_2$ with $c_1 = \frac{q}{p} \sum_k f_k f_{-k}$ and $c_2 = \frac{q}{p} f_0$. Recall that $g_0 = \lambda$. The second component is determined by $\sum_m f_{mq}g_{-mp} \exp(2\pi imqu)$. The $\lambda$ dependent factor of the modulus of the Fourier coefficients is $|\sin(mp\pi\lambda)|$. The latter is a periodic function of $\lambda$ with $p+1$ zeroes for all $m$. Or, put differently, with at least $p$ folds for all $m$. Thus the number of stability pockets is determined by the first and largest Fourier coefficients, $f_1$ and $f_{-1}$, and thus by the map $(\delta, \mu) = (\sigma - \eta^2 c_1 - \eps \eta \lambda c_2, \eps \eta c_3 |\sin(p\pi\lambda)|)$. Therefore the $(p,q)$-tongues have a chain of $p$ stability pockets. To describe their shape in detail one would need the remaining Fourier coefficients.\commentaar{hm!}
\end{proof}
%
%--------------------------------------------------------------------------------

%--------------------------------------------------------------------------------
\subsection{Generalizations}\label{sec:gens}
Let us now generalize the forcing $g_{\lambda}$ to a function $h_{\lambda}$ such that still $h_0 = 0$ and $h_1 = 1$, so that $\lambda = 0$ corresponds to winter and $\lambda = 1$ to summer. But for $\lambda \in (0,1)$, $h_{\lambda}$ is not necessarily a block function. For simplicity we assume that $h_{\lambda}$ is $\cinf$. Just like the Fourier coefficients of $g_{\lambda}$, those of $h_{\lambda}$ can be regarded as parameterizations of a closed curve in $\C$ with endpoints at zero, the parameter being $\lambda$. Now for $k \neq 0$ the $\lambda$ dependent factor of $g_k(\lambda)$ is $\big(\exp(-2\pi ik\lambda) - 1\big)$, which passes through zero for $\lambda = \frac{n}{k}$ for $n \in \{0,1,\ldots,k\}$. When we perturb $g_{\lambda}$ to $h_{\lambda}$ subject to the condition that $h_0 = 0$ and $h_1 = 1$, the Fourier coefficients again parameterize closed curves in $\C$ with endpoints at zero, but in general not passing through zero for $\lambda \in (0,1)$. Thus we will still have stability pockets for the more general system, but the chains of stability pockets will open up to a single pocket, see the image of the $(2,1)$-tongue in figure \ref{fig:mtpocket}.
%
%--------------------------------------------------------------------------------

%--------------------------------------------------------------------------------
\section{Proofs}\label{sec:proofs}
This section contains the proofs of the theorems in the previous sections. We first give a short review of approximating a Poincar\'e map and then apply this to our coupled oscillators. In a last section we discuss how to deal with a nonsmooth forcing.

%--------------------------------------------------------------------------------
\subsection{A vectorfield approximation of the Poincar\'e map}\label{sec:vvapprox}
In this section we consider vectorfields on the phase space $\R^2$ as lifts of vectorfields on the two torus. Our aim is to construct a vectorfield approximation of the Poincar\'e map. Where the latter is the lift of the Poincar\'e map on the two torus. We will work in the $\cinf$ context unless explicitly stated otherwise.

First we define a parameter dependent differential equation on $\R^2$. Let $f_{\mu}$ be a function $\R^2 \to \R$, 1-periodic in both arguments and assume $f$ has a Taylor-Fourier expansion
\begin{equation*}
f_{\mu}(x, y) = \sum_{k,l,m} f_{k,l,m} \, \mu^m \e^{2\pi i(kx+ly)}
\end{equation*}
with coefficients $f_{k,l,m}$. Now consider the differential equation
\begin{equation}\label{eq:tpo2}
\system{\dot{x} &= \wpq + \delta + \mu f_{\mu}(x, y) \\ \dot{y} &= 1}
\end{equation}
depending on parameters $\delta$ and $\mu$. The first will be interpreted as a detuning and the second as the strength of the nonlinearity. The Poincar\'e map $F$ of this system is defined as $(F(x), 1) = \Phi_1(x,0)$, where $\Phi_t$ is the flow over time $t$ of equation \eqref{eq:tpo2}.\commentaar{P-afb is niet uniek, maar 'alle' P-afbn zijn equivalent via de stroming van de diffvgl}

\textbf{Computing the vectorfield approximation.} By a sequence of near identity transformations we put system \eqref{eq:tpo2} in a form such that the $y$ dependence in the first component becomes trivial, at least up to a certain order. Here we use perturbation theory so we have to be precise about the size of various components. The grading we use will come from the parameters only. The reason is that the vectorfield is periodic in $x$ and $y$ therefore we set $\degree(x) = \degree(y) = 0$. Furthermore we set $\degree(\mu) = \degree(\delta) = 1$. With these definitions the function $f$ has a formal Fourier-Taylor expansion in $x$, $y$ and $\mu$
\begin{equation*}
f_{\mu}(x, y) = \sum_{k,l,m} f_{k,l,m} \, \mu^m \e^{2\pi i(kx+ly)}
\end{equation*}
with Fourier-Taylor coefficients $f_{k,l,m}$. Finally we define the vectorfield
\begin{align*}
X   &= X_0 + X_1 + \cdots + X_{n} + \cdots\\
X_0 &= \wpq \bvv{x} + \bvv{y}\\
X_j &= \sum_{k,l,|m|=j} f_{k,l,m} \, \mu^m \e^{2\pi i(kx+ly)} \bvv{x}
\end{align*}
Now we apply a standard normalization procedure to obtain a normal form of the vectorfield $X$. The result is formulated in the next proposition.
\begin{proposition}\label{pro:vvapprox}
By a sequence of $n$ near-identity transformations the vectorfield $X$ is transformed into $\tilde{X}$ with
\begin{align*}
\tilde{X}   &= X_0 + \tilde{X}_1 + \cdots + \tilde{X}_n + \cO(|\mu|^{n+1})\\
\tilde{X}_j &= \sum_{|m|=j,pk+ql=0} \tilde{f}_{k,l,m} \, \mu^m \e^{2\pi i(kx+ly)} \bvv{x}
\end{align*}
The new coefficients $\tilde{f}_{k,l,m}$ depend on the original coefficients in a complicated way. One exception being $\tilde{X}_1$ for which $\tilde{f}_{k,l,m} = f_{k,l,m}$. The transformed vectorfield up to order $n$ only contains \emph{resonant} terms, that is in general $f_{k,l,m} \neq 0$ only if the index $(k,l,m)$ satisfies the \emph{resonance condition} $pk + ql = 0$.
\end{proposition}

\begin{proof}
This will only be a sketch of the proof, for more details see \cite{hwb2009a}. The key idea of the proof is that every $\cinf$ near identity coordinate transformation can be approximated as closely as needed by the flow of a $\cinf$ vectorfield \cite{ft1974}. Since we work in the context of $\cinf$ vectorfields on the two torus (or lifts thereof) the transformation has to respect this property. The flow of a $\cinf$ vectorfield on the two torus is such a transformation. Since these vectorfields form a Lie algebra we can use rather standard normal form theory. Our aim is to get rid of the time ($y$) dependence without transforming time, therefore we apply an asymmetric transformation generated by the vectorfield $Y=g(x,y)\bvv{x}$ so that the new $x$ depends on $y$ but not the other way around. The procedure is inductive by degree. Let $\cX$ be the set of $\cinf$ vectorfields on the two torus depending on one or more small parameters $\mu$. We define a grading by the degree of $\mu$: $\cX_k \subset \cX$ is the set of vectorfields of degree $k$ or higher in $\mu$. In the normal form procedure we frequently use the commutator $[\cdot,\cdot]$ of vectorfields, then for $X_m \in \cX_m$ we have $[X_m, X_n] \in \cX_{m+n}$. Furthermore if $Y_m \in \cX_m$ is a fixed vectorfield and $\ad Y_m: X \mapsto [Y_m, X]$ then $(\ad Ym)^k (X_n) \in \cX_{km+n}$. Because of these relations we may normalize for increasing degree.

Now suppose we have normalized up to degree $n-1$, then take $Y$ of degree $n$. The transformation acts on the vectorfield $X$ as $\expad Y$, namely
\begin{equation*}
\expad Y (X) = \e^{\ad Y} (X) = X_0 + X_1 + \cdots + X_{n-1} + X_n + \ad Y(X_0) + \hot
\end{equation*}
where $\ad Y(x) = [Y,X]$. As usual in normal form theory we try to solve $X_n + \ad Y(X_0) = 0$ for $g$. Let $g_{k,l}$ be the Fourier coefficients of $g$ then on the level of coefficients the equation becomes $qf_{k,l} - 2\pi i (pk + ql) g_{k,l} = 0$. Thus we set $g_{k,l} = \frac{qf_{k,l}}{2\pi i (pk + ql)}$ provided that $pk + ql \neq 0$. This implies that in the normal form the so called \emph{resonant terms} with coefficient $f_{k,l}$ satisfying the resonance condition $pk + ql = 0$ are retained.
\end{proof}

Assuming we have normalized the vectorfield up to sufficiently high order, we only keep terms up to order $n$ by truncation. As a last step we again change coordinates: $u = x - \tfrac{p}{q}y$ and $v = y$. Then the new vectorfield $Z$ becomes
\begin{equation*}
Z = \Big(\delta + \sum_{j=1}^n \sum_{|m|=j,pk+ql=0} \tilde{f}_{k,l,m} \, \mu^m \e^{2\pi iku}\Big) \bvv{u} + \bvv{v}
\end{equation*}
\commentaar{waar komt $\delta$ ineens vandaan?}
Indeed, if $pk + ql = 0$ then changing to coordinates $u$ and $v$ we get $kx + ly = k(u + \tfrac{p}{q}y) + ly = ku + \tfrac{pk + ql}{q} y = ku$.

\textbf{Interpretation of vectorfield $Z$.} Now let $\tilde{X}$ be the truncated normal form of $X$ and $Z$ be the vectorfield defined above. The corresponding flows are denoted by respectively $\tilde{\Phi}$, $\Phi$ and $\Psi$. Let $F$ be the Poincar\'e map of $\Phi$ defined as $(F(x),1) = \Phi_1(x,0)$. In a similar way we define $\tilde{F}$, then this $\tilde{F}$ is an approximation up to order $n$ of $F$. Furthermore let $\psi$ be the flow of the first component of $Z$. Then we have $\Psi_s(u, 0) = (\psi_s(u), s)$ but also $\Psi_s(u, 0) = (\Pi_x\tilde{\Phi}_s(u, 0) - s\tfrac{p}{q}, s)$, where $\Pi_x$ is the projection $\Pi_x(x,y) = x$. This means that $\tilde{F}^q(u) = \Pi_x\tilde{\Phi}_q(u, 0) = \psi_q(u)+p$. Thus $\tilde{F^q} - p$ is equal to the time $q$ flow of the first component of $Z$. Therefore we call $Z$ a vectorfield approximation of $F$.

\textbf{Using the vectorfield approximation.} The vectorfield $Z$ approximates the Poincar\'e map of the vectorfield $X$ in \eqref{eq:tpo2} but there is no equivalence. Therefore we have to be careful when drawing conclusions about the dynamics of \eqref{eq:tpo2} from analysis of the vectorfield $Z$. Here we will be mainly interested in stationary points of $Z$. A first observation is that hyperbolic stationary points of\commentaar{eerste component} $Z$ correspond to relative equilibria of $\tilde{X}$ and thus to periodic orbits of $X$ (provided that the difference between $X$ and $\tilde{X}$ is small enough). Stationary points of $Z$ satisfy equation
\begin{equation}\label{eq:statpt}
0 = \delta + \sum_{j=1}^n \sum_{|m|=j,pk+ql=0} \tilde{f}_{k,l,m} \, \mu^m \e^{2\pi iku}
\end{equation}
Given $p$ and $q$, the right hand side is a $\frac{2\pi}{q}$ periodic function. Therefore solutions, if they exist, come in $q$ pairs. Now suppose solutions exist, then upon varying parameters $\mu$ they may disappear in tangencies. Assuming for simplicity a single parameter $\mu$, we find a wedge shaped region in parameter space defined by $\gamma_1 \mu < \delta < \gamma_2 \mu$ where solutions exist. This inequality only holds near $(\delta,\mu) = (0,0)$ and the constants $\gamma_1$ and $\gamma_2$ depend on the Taylor-Fourier coefficients $\tilde{f}_{k,l,m}$. Dynamically speaking the boundaries of the wedge are curves of saddle-node bifurcations. These form the familiar stability tongues or Arnol'd tongues. Since saddle-node bifurcations in one parameter families persist under small perturbations, the stability tongues of vectorfield $Z$ approximate those of vectorfield $X$.

For each combination of $p$ and $q$ (positive and relative prime) we find a tongue. The one for $(p, q) = (1, 1)$ is called the main tongue. If we assume for simplicity that the Taylor-Fourier coefficients rapidly decrease for increasing $k$ and $l$ then the leading terms in equation \eqref{eq:statpt} are
\begin{equation}\label{eq:statptapprox}
0 = \delta + \gamma \mu \sin(2\pi u + \chi)
\end{equation}
where $\gamma$ and $\chi$ are determined by $\tilde{f}_{1,-1,1}$ and $\tilde{f}_{-1,1,1}$. This little calculation at least shows that the tip of the tongue is a straight cone.

Similarly the tongues for $(p, q) = (1, q)$ are determined by an equation whose leading terms are as in equation \eqref{eq:statptapprox}, but now $\gamma$ and $\chi$ are determined by $\tilde{f}_{q,-1,1}$ and $\tilde{f}_{-q,1,1}$. Again this shows that the tip of the tongue is a straight cone, but the angle of the cone decreases with increasing $q$.

\commentaar{wanneer naar hogere graad? als f een enkele sin is?}
\begin{remark}\label{rem:vvapprox}
We collect some remarks about the vectorfield approximation.\commentaar{nog wat cryptisch in deze vorm...}
\begin{enumerate}[i)]\itemsep 0pt
\item The transformation $u = x - \tfrac{p}{q}y$ and $v = y$ is not just useful but originates from the following. By construction the truncated vectorfield $\tilde{X} = X_0 + \tilde{X}_1 + \cdots + \tilde{X}_n$ commutes with $X_0$ implying that their flows $\tilde{\Phi}$ and $\Phi^0$ also commute. In other words the flow of $X_0$ generates a symmetry group of $\tilde{X}$. By switching to 'co-moving' coordinates, thus by the transformation $\Phi^0_t$, the vectorfield $\tilde{X}$ transforms to $\tilde{X} - X_0$. Therefore stationary points of $\tilde{X} - X_0$ correspond to periodic orbits of $\tilde{X}$. Because of this property such stationary points are called \emph{relative equilibria}.\commentaar{hoort die niet onder 'interpretatie'?}
\item In the normal form of the seasonal oscillator we will only need terms up to degree two. Then we have to compute a few higher order ($\hot$) terms. Let $X$ and $Y$ be as in the proof, but now $Y$ is of degree $1$ then the terms we need are
\begin{equation*}\label{eq:hot}
\expad Y (X) = X_0 + X_1 + \ad Y(X_0) + \ad Y(X_1) + \frac{1}{2} (\ad Y)^2 (X_0).
\end{equation*}
\end{enumerate}
\end{remark}
%--------------------------------------------------------------------------------

%--------------------------------------------------------------------------------
\subsection{Approximation of Poincar\'e map of the seasonal oscillator}\label{sec:apmapso}
We apply proposition \ref{pro:vvapprox} to the differential equation of the seasonal oscillator. The starting point is the general form of the lift of equation \eqref{eq:season}
\begin{equation*}
\system{\dot{x} &= \omega + \eta f(x) + \eps g_{\lambda}(y) \\ \dot{y} &= 1}
\end{equation*}
We normalize this vectorfield to obtain an approximate Poincar\'e map. We assume that $f$ and $g$ both have a Fourier series with coefficients $f_k$ and $g_k$. Then the vectorfield is
\commentaar{wat doen we met $f_0$ en $g_0$?}
\begin{align*}
X   &= X_0 + X_1 + X_2\\
X_0 &= \frac{p}{q} \bvv{x} + \bvv{y}\\
X_1 &= (\eta f(x) + \eps g(y)) \bvv{x} = \sum_{k,l} f_{k,l} \e^{2\pi i(kx+ly)} \bvv{x}\\
X_2 &= \delta \bvv{x}
\end{align*}
where $f_{k,0} = \eta f_k$ and $f_{0,l} = \eps g_l$. For both $k \neq 0$ and $l \neq 0$ we set $f_{k,l} = 0$. In view of the resonance condition $pk+ql=0$ for a term $\exp(2\pi i(kx+ly)) \bvv{x}$, see proposition \ref{pro:vvapprox}, there will be no resonant terms of degree 1, therefore we set $\degree(\delta) = 2$ and $\degree(\eps) = \degree(\eta) = 1$. Note that $\omega = \frac{p}{q} + \delta$.

The proof of theorem \ref{the:season} consists of determining the vectorfield of degree $1$ and computing terms up to degree $2$.

\begin{proof}
Following the proof of proposition \ref{pro:vvapprox} we take a vectorfield $Y$ of degree $1$ and we try to solve $X_1 + \ad Y(X_0) = 0$ for the Fourier coefficients of $Y$. The degree is determined by the parameters $\delta$, $\eps$ and $\eta$. Let $Y = a(x,y) \bvv{x}$ and $a$ has Fourier coefficients $a_{k,l}$. Then on the level of Fourier coefficients we get $a_{k,0} = \eta \frac{q}{p} \frac{1}{2\pi ik} f_k$ if $k \neq 0$ and $a_{0,l} = \eps \frac{1}{2\pi il} g_l$ if $l \neq 0$. Since we have solved equation $X_1 + \ad Y(X_0) = 0$ we have $\ad Y(X_1) + \frac{1}{2} (\ad Y)^2 (X_0) = -\frac{1}{2} \ad Y(X_1)$. Thus we obtain up to degree two
\begin{align*}
\tilde{X} &= \expad Y(X) = X_0 + X_1 + \ad Y(X_0) + \ad Y(X_1) + \frac{1}{2} (\ad Y)^2 (X_0) + \hot\\
          &= X_0 + \frac{1}{2} \ad Y(X_1) + \hot
\end{align*}
As a shorthand write $X_1 = f \bvv{x}$, $Y = a \bvv{x}$ and $\ad Y(X_1) = h \bvv{x}$ then $h = a \partiel{f}{x} - f \partiel{a}{x}$. From this expression we select the resonant terms, that is the $h_{k,l}$ satisfying $pk + ql = 0$, then we are left with
\begin{align*}
h(x,y) &= -2 \eta^2 \frac{q}{p} \sum_k f_k f_{-k} - 2 \eps \eta \frac{q}{p} \sum_{pk+ql=0} f_k g_l \e^{2\pi i(kx+ly)}\\
       &= \sigma - \frac{q}{p} \Big(\eta^2 \sum_k f_k f_{-k} + \eps \eta \sum_m f_{mq} g_{-mp}(\lambda) \e^{2\pi im(qx-py)} \Big)
\end{align*}
Setting $u = x - \frac{p}{q}y$ and $v = y$ the vectorfield approximation of the Poincar\'e map becomes
\begin{equation*}
\system{\dot{u} &= \delta - \eta^2 \frac{q}{p} \sum_k f_k f_{-k} - \eps \eta \sum_m f_{mq} g_{-mp}(\lambda) \e^{2\pi imqu}\\ \dot{v} &= 1}
\end{equation*}
\end{proof}

With this result we find the stability tongues of the vectorfield approximation. If $(p,q) = (1,1)$ then the tongue boundaries follow from solving equation
\begin{equation*}
0 = \delta - \eta^2 \sum_k f_k f_{-k} - \eps \eta \sum_l f_{-qk} g_k \e^{2\pi iqku}
\end{equation*}
for $u$, see section \ref{sec:vvapprox}. 

%--------------------------------------------------------------------------------

%--------------------------------------------------------------------------------
\subsection{The Fourier coefficients of the forcing}\label{sec:fcf}
The forcing $g_{\lambda}$ of the seasonal oscillator, see equation \eqref{eq:block}, is a piecewise constant function and therefore not $\cinf$ as required in proposition \ref{pro:vvapprox} on the approximation of the Poincar\'e map. In order to get a smooth approximation of $g_{\lambda}$ we use convolution with a so called Schwartz function. For this smooth approximation we find the Fourier coefficients. We are in particular interested in the dependence on the parameter $\lambda$. The proofs of the following proposition and lemmas are found by straightforward arguments and computations.

\begin{proposition}\label{pro:fcf}
Let $\phi$ be a normalized Schwartz function and let $\hat{\phi}$ be its Fourier transform. Then the Fourier coefficients of the convolution $\phi * g_{\lambda}$ are $\hat{\phi}(k) \cdot g_k$ where $g_0 = 0$ and $g_k = \frac{i}{2\pi k} \big(\e^{-2\pi ik\lambda} - 1\big)$ if $k \neq 0$. In particular we have $|g_k| = \frac{1}{2\pi k} |\sin(\pi k \lambda)|$.
\end{proposition}

Let $\phi$ be a normalized Schwartz function, that is
\begin{enumerate}[i)]\itemsep 0pt
\item for all positive integers $m$ and $n$ we have $\sup_{\R} |x^m \phi^{(n)}(x)| < \infty$,
\item $\int_{\R} \phi(x)\,dx = 1$.
\end{enumerate}
The familiar Gauss function $\phi(x) = \frac{1}{\sqrt{\pi}} \e^{-x^2}$ is an example of such a function. For all $\alpha > 0$ the function $\phi_{\alpha}(x) = \alpha \phi(\alpha x)$ is again a normalized Schwartz function and the larger $\alpha$, the closer $\phi_{\alpha} * g_{\lambda}$ is to $g_{\lambda}$.\commentaar{in welke zin?}

\begin{lemma}
Let $g$ be a possibly non-smooth 1-periodic function with a Fourier series such that $g(x) = \sum_k g_k \exp(2\pi kx)$ and let $\phi$ be a normalized Schwartz function with Fourier transform $\hat{\phi}$. Then $\phi * g$ is a smooth 1-periodic function which has a Fourier series with coefficients $\hat{\phi}(k) \cdot g_k$.
\end{lemma}

This lemma shows that the results in proposition \ref{the:pockets} do not depend on the choice of the Schwartz function $\phi$ since the zeroes of $\hat{\phi}(k) \cdot g_k(\lambda)$, as a function of $\lambda$, are exactly those of $g_k(\lambda)$.

By an elementary calculation we immediately find the Fourier coefficients of $g_{\lambda}$ and a smooth approximation of $g_{\lambda}$.
\begin{lemma}\label{lem:fcg}
The Fourier coefficients $g_k$ of the function $g_{\lambda}$ as defined in equation \eqref{eq:block}, are $g_0 = \lambda$ and $g_k = \frac{i}{2\pi k} \big(\e^{-2\pi ik\lambda} - 1\big)$ if $k \neq 0$. The Fourier transform of $\phi_{\alpha}$ is $\hat{\phi}_{\alpha}(y) = \exp(-\frac{\pi^2 y^2}{\alpha})$. 
\end{lemma}

Our main interest is in the $\lambda$ dependence of the Fourier coefficients. So far we found $\hat{\phi}(k) \cdot g_k = \exp(-\frac{\pi^2 k^2}{\alpha}) \frac{i}{2\pi k} \big(\exp(-2\pi ik\lambda) - 1\big)$ if $k \neq 0$, from which we infer that
\begin{equation*}
 |\hat{\phi}(k) \cdot g_k| = \e^{-\frac{\pi^2 k^2}{\alpha}} \frac{1}{2\pi k} \sqrt{2(1-\cos(2\pi k \lambda))} = \e^{-\frac{\pi^2 k^2}{\alpha}} \frac{1}{\pi k} |\sin(\pi k \lambda)|.
\end{equation*}

%--------------------------------------------------------------------------------

%--------------------------------------------------------------------------------


\begin{thebibliography}{1}

\bibitem{arn1983}
V.I. Arnol'd.
\newblock {\em Geometrical Methods in the Theory of Ordinary Differential
  Equations}.
\newblock Springer-Verlag, New York, 1983.

\bibitem{hwb2009a}
H.W. Broer.
\newblock Normal forms in perturbation theory.
\newblock In Robert~A. Meyers, editor, {\em Encyclopedia of Complexity and
  Systems Science}, pages 6310--6329. Springer, New York, 2009.

\bibitem{bl1995}
H.W. Broer and M.~Levi.
\newblock Geometrical aspects of stability theory for hill's equations.
\newblock {\em Arch. Rational Mech. Anal.}, 131:225--240, 1995.

\bibitem{bst1998}
H.W. Broer, C.~Sim\'o, and J.C. Tatjer.
\newblock Towards global models near homoclinic tangencies of dissipative
  diffeomorphisms.
\newblock {\em Nonlinearity}, 11:667--770, 1998.

\bibitem{smhb}
Christoph Schmal, Jihwan Myung, Hanspeter Herzel, and Grigory Bordyugov.
\newblock A theoretical study on seasonality.
\newblock {\em Front. Neurol. 6:94. doi: 10.3389/fneur.2015.00094}, 2015.

\bibitem{ft1974}
F.~Takens.
\newblock Singularities of vectorfields.
\newblock {\em Publications de math\'ematiques de l'IHES}, 43:47--100, 1974.

\bibitem{ps1928}
B.~van~der Pol and M.J.O. Strutt.
\newblock On the stability of the solutions of mathieu's equation.
\newblock {\em Philosophical Magazine}, 5:18--38, 1928.

\end{thebibliography}
\end{document}